\documentclass[11 pt,leqno]{amsart}

\usepackage{amsmath}
\usepackage{amssymb}
\usepackage{amsfonts}
\usepackage{amsthm}
\usepackage{hyperref}
\usepackage{tikz}
\usepackage{enumerate}

\usepackage{bbm}
\usepackage[mathscr]{eucal}

\def\bbone{{\mathbbm 1}}

\newcommand{\Eann}{ E_{\mathrm{ann}}}

\newcommand{\sA}{\mathscr A}

\newcommand{\Be}{\begin{equation}}
\newcommand{\Ee}{\end{equation}}
\newcommand{\Bea}{\begin{align}}
\newcommand{\Eea}{\end{align}}
\newcommand{\Beas}{\begin{align*}}
\newcommand{\Eeas}{\end{endalign*}}
\newcommand{\Benu}{\begin{enumerate}}
\newcommand{\Eenu}{\end{enumerate}}
\newcommand{\Bi}{\begin{itemize}}
\newcommand{\Ei}{\end{itemize}}

\newcommand{\ci}[1]{_{{}_{\!\scriptstyle{#1}}}}

\def\bbone{{\mathbbm 1}}

\def\intslash{\rlap{\kern  .32em $\mspace {.5mu}\backslash$ }\int}
\def\qsl{{\rlap{\kern  .32em $\mspace {.5mu}\backslash$ }\int_{Q_x}}}
\def\Re{\operatorname{Re\,}}

\def\vth{\vartheta}

\def\N{\mathbb N}

\def\emph#1{{\it #1 }}

\def\inn#1#2{\langle#1,#2\rangle}

\def\meas{{\text{\rm meas}}}
\def\ud{{\text{\,\rm d}}}

\def\lc{\lesssim}

\def\eps{\varepsilon}

\def\ka{\kappa}
            
\def\la{\lambda}             \def\La{\Lambda}

\def\om{\omega}

\def\fA{{\mathfrak {A}}}

\def\bbN{{\mathbb {N}}}

\def\bbR{{\mathbb {R}}}

\def\bbZ{{\mathbb {Z}}}

\def\cC{{\mathcal {C}}}

\def\cF{{\mathcal {F}}}

\def\cS{{\mathcal {S}}}

\newcommand{\lri}[1]{}

\theoremstyle{plain}
\newtheorem{thm}{Theorem}[section]
\newtheorem{lemma}[thm]{Lemma}

\theoremstyle{remark}
\newtheorem{remark}[thm]{Remark}

\numberwithin{equation}{section}

\begin{document}

\title[A lower bound]
{A lower bound 
for a variation norm operator  associated with  circular means}

\author[D. Beltran, A. Carbery, L. Roncal, A. Seeger]{David Beltran \ \  Anthony Carbery \ \ Luz Roncal \ \  Andreas Seeger }

\address{David Beltran, Departament d’An\`alisi Matem\`atica, Universitat de Val\`encia, Av. Vicent Andrés Estellés 19, 46100 Burjassot, Spain}
\email{david.beltran@uv.es}

\address{Anthony Carbery, School of Mathematics and Maxwell Institute for Mathematical
Sciences, University of Edinburgh, James Clerk Maxwell Building, Peter Guthrie Tait
Rd, Kings Buildings, Edinburgh EH9 3FD, Scotland}

\email{A.Carbery@ed.ac.uk}

\address{Luz Roncal, BCAM – Basque Center for Applied Mathematics, 48009 Bilbao, Spain, and
Ikerbasque, Basque Foundation for Science,
48011 Bilbao, Spain
and
Universidad del Pa\'is Vasco / Euskal Herriko Unibertsitatea,
Apartado 644, 48080 Bilbao, Spain}
\email{lroncal@bcamath.org}

\address{Andreas Seeger, Department of Mathematics, University of Wisconsin--Madison, Madison, WI 53706, USA}
\email{seeger@math.wisc.edu}

 \dedicatory{Dedicated to Professor Hans Triebel on the occasion of his 90th birthday}

 \subjclass[2020]{42B15, 42B20, 42B25, 42B35}

 \keywords{Circular averages, variation norm bounds, square functions, ball multipliers, Besicovitch sets}

\begin{abstract}
We  prove that a local $L^p(V_2)$ variation norm estimate fails for circular means in two dimensions, and quantify this failure by proving lower bounds for functions of exponential type. This is related to lower bounds for Fourier multipliers supported on annuli, of the type considered by C\'ordoba. 
\end{abstract}

\date{\today}
\maketitle

\section{Introduction}
Consider the circular averages 
\[A_t f(x) = \frac1{2\pi} \int_0^{2\pi}  f(x_1-t\cos\alpha, x_2-t\sin\alpha) \, \ud\alpha \] 
for  functions in $L^p(\bbR^2)$.

Let $I=[1,2]$ and let $V^I_r A$ be the operator obtained from taking the $r$-variation semi-norm with respect to the $t$ variable of the map $t\mapsto A_t f$ over the interval $I$, that is
$$
V_r^I A f(x)
:= \sup_{N \in \N} \,\, \sup_{\substack{t_1 < \cdots < t_N \\ t_j \in I}} \Big(\sum_{j=1}^{N-1} |A_{t_{j+1}}f(x) - A_{t_j}f(x) |^r \Big)^{1/r}.
$$
For the family of spherical means the  variation operator $V^I_r A$ and its global analogue have been  studied in \cite{JSW} and then more recently in \cite{BeltranOberlinRoncalSeegerStovall}.  
The $L^p$-boundedness for $V^I_r A $ for some $r$  implies the $L^p$-boundedness of the local circular maximal operator \cite{Bourgain1986} (with dilations restricted to $I$); therefore boundedness  fails for $p\le 2$.
It is known \cite{JSW}  that $V_r^IA$ maps $L^p$ into itself for $2<p\le 4$,  $r> 2$ 
and for $p>4$, $r>p/2$; moreover $L^p$-boundedness fails for $r<\max\{ 2, p/2\}$. 
A key question, namely the $L^p$-boundedness for the local variation operator $V_2^I A$ in the range $2<p\le 4$,  remained open. We remark  that the usual restriction $r>2$ in variation-norm bounds is related to L\'epingle's theorem \cite{Lepingle1976} which applies to the {\it global} variation norm operator {(that is, when $I$ is replaced by $(0,\infty)$)}. However, this does not suggest any necessary condition for the local variant  $V_r^I A$.

It turns out that $V^I_2A$ is not bounded on any $L^p$ space.
\begin{thm} \label{thm:main} For all $p\ge 1$, 
\[\sup_{\substack{f\in \cS(\bbR^2)\\ \|f\|_{p} \le 1 } }\big\|V_2^I A f\|_{L^p(\bbR^2)} =\infty. \]
\end{thm}

As pointed out above, this is well-known for $p\notin (2,4]$. In the interesting range $2<p\le 4$ we will obtain the result as an immediate consequence of a more quantitative version which we now present. Motivated somewhat by the presentation in \cite[\S I.1.]{Triebel1983}, we test $V^I_2A$ on $L^p$-functions of exponential type, i.e. functions whose  Fourier transform is supported in large balls. 

For $\la>1$ we let $E(\la)$ be the space of all tempered distributions whose Fourier transform is supported in $\{\xi:|\xi|\le\la\}$.
Define 
 \Be \label{eqn:Bp}B_p(\la) =\sup\,\big \{ \|V_2^I A f\|_{L^p(\bbR^2)} : \,\|f\|_{L^p(\bbR^2) } \le 1, \,\,f\in E(\la)
 \big\}.
 \Ee
Clearly $B_p(\la)$ is finite and increasing in $\la$; note that for all $1 \leq p \leq \infty$
\begin{equation}\label{eq:finiteness Bp}
\|V_2^IA f\|_p \lc \sup_{1\le t\le 2} \|\partial_t A_t f\|_p \lc \la \|f\|_p \text{  \quad for $f\in E(\la) $, } 
\end{equation}
which implies $B_p(\la)\le C\la$. More refined arguments in \cite{JSW} 
(related  to a square function estimate in \cite{Carbery1983} essentially 
via \cite{KanekoSunouchi}) yield \[
B_p(\la) \lc (\log \la)^{C} \quad  \text{  for $2\le p\le 4$,}
\] for a suitable positive exponent  $C$. 
Our quantification of Theorem \ref{thm:main} is 
\begin{thm}\label{thm:kakeya} Let $p>2$. 
Then there are constants  $\la_0>1$, $c>0$ depending on $p$ such that 
\[  B_p(\la) \ge c (\log\la)^{\frac 12-\frac 1p} \]
for all $\la>\la_0$.
\end{thm}
We are only interested in the range $2<p\le 4$ as larger lower bounds are known for $p>4$ (see \cite{JSW}, \cite{BeltranOberlinRoncalSeegerStovall}). 
In the proof  of the theorem we  reduce matters to  lower bounds 
for multipliers of the form $\chi(|\xi|-\la)$  for large $\la$, with nonnegative $\chi \in C^\infty_c(\bbR)$.  C\'ordoba  \cite{Cordoba1979} showed that for $4/3\le p\le 4$ the operator norm of the corresponding convolution operators is $O( (\log\la)^{|1/p-1/2|})$,  and this also matches the lower bound in Lemma \ref{lem:aux2} below. 

\subsubsection*{Notation} For  nonnegative  quantities $a,b$, we  write $a \lesssim b$ or $a\lesssim_L b$
 to indicate $a \leq C b$ for some constant $C$ which may depend  on some list $L$.

\section{Proof of Theorem \ref{thm:kakeya}}
In order to establish Theorem \ref{thm:kakeya}, we prove several auxiliary lemmas which link the problem to a class of radial Fourier multipliers. Our proof is inspired by a result of Kaneko and Sunouchi \cite{KanekoSunouchi}, who proved the pointwise equivalence of two global square functions first occurring in work by Stein: one associated with Bochner--Riesz means \cite{SteinActa58} (see also \cite{Carbery1983, ChristPAMS, Seeger-crelle1986, LeeRogersSeeger2014}), and one associated with spherical means \cite{Stein1976, SteinWainger-Bull78}.  

For $\la_1<\la_2$, let $\Eann(\la_1,\la_2)$ 
be the space of tempered distributions whose Fourier transform is compactly supported in $\{\xi:\la_1\le |\xi|\le \la_2\}$. Let $\sigma$ denote the normalized surface measure on the unit circle $S^1$. Also, given a distribution $\mu$, we define the dilate $\mu_t:= t^{-2} \mu(t^{-1} \cdot)$,  in the sense of distributions. 

\begin{lemma} \label{lem:LpL2chi} Let $\la\ge 1$ and $p>2$. Then 
\Be\label{eqn:localizedLpL2}  \Big\|  \Big(\int_{5/4}^{7/4} |(\chi \sigma)_t* g |^2 \ud t \Big)^{1/2} \Big\|_p  \lc_p \la^{-1/2} (B_p(2\la)+1) \|g\|_p\Ee
for all $g\in \Eann(\la/4,\la)$ and all $\chi\in C_c^\infty (\bbR^d\setminus \{0\})$.
\end{lemma}

\begin{proof}
Let $\upsilon\in C^\infty_c(\bbR)$ be supported in $[1,2]$ and such that $\upsilon(t)=1$ in a neighborhood of  $[5/4, 7/4]$. Define
\Be \label{eqn:sAdef}\sA f(x,t):=\upsilon (t) A_tf(x)\Ee
{and consider the associated variation-norm operator $V_2^I \sA f(x)$.} 
It is easy to see {via the triangle inequality and the mean value theorem} that
\[ V_2^I \sA f(x)\le \|\upsilon\|_\infty V_2^I Af(x) 
+\|\upsilon'\|_\infty \sup_{t\in[1,2]}|A_t f(x)|,\]
and by using Bourgain's circular maximal theorem we get for $2<p \leq \infty$
\[\| V_2^I \sA f\|_p \lc(B_p(2\la)+C_p) \|f\|_p\\
\quad \text{ if $f\in E(2\la)$.}
\]
Let $\{\Lambda_j\}_{j \in \mathbb{Z}}$ be  a standard dyadic frequency decomposition $\{\La_j\}_{j\in \bbZ}$
in the  $t$ variable (so that $\La_j $ localizes to frequencies $\tau$ with $2^{j-1}\le |\tau|\le 2^{j+1}$).
Then the Besov space  seminorm for  $a\in \dot B^{1/2}_{2,\infty} $ is given by $\sup_{j\in \bbZ} 2^{j/2} \|\La_j a\|_2$. By the continuous embedding $V_2\hookrightarrow  \dot B^{1/2}_{2,\infty}$ (\cite{BerghPeetre}) we have  
\begin{align*} & 2^{j/2} \| \La_j \sA f \|_{L^p(L^2) }  \le \|\sup_{j>0} 2^{j/2} | \La_j \sA f| \|_{L^p(L^2)} 
\le  \| \sA f\|_{L^p(\dot B_{2,\infty}^{1/2} )}
\\&\notag \lc \| V_2^I \sA f\|_p
\lc (B_p(2\la)+C_p) \|f\|_p \quad \text{ for $f\in E(2\la)$}.
\end{align*}
We use this bound for  $2^{-10}\le \la/2^j\le 2^{10}$. For $f\in \Eann(\la/8, 2\la)$ we  refer to an error estimate involving a negligible upper bound  for other values of $j$ to \cite[Lemma 2.5]{BeltranOberlinRoncalSeegerStovall} and we get 
 \[ \sum_{\substack {j\ge 0 :\\
 2^j\notin [2^{-10}\la, 2^{10}\lambda]}} \|\La_j \sA f \|_{L^p(L^2)} \lc_N\la^{-N} \|f\|_p \text{ if } f\in L^p \cap \Eann(\la/8, 2\la) .\] 
Consequently we have 
\[  \|  \sA f \|_{L^p(L^2) }  \lc \la^{-1/2}  (B_p(2\la)+C_p+1) \|f\|_p, \quad\text{for $f\in L^p\cap \Eann(\la/8,2\la)$}.\]
Since $\upsilon(t)=1$ on $[5/4,7/4]$ we obtain,
for $f\in L^p\cap \Eann(\la/8,2\la)$, 
\begin{equation}\label{eq:2lambda}
    \Big\|  \Big(\int_{5/4} ^{7/4} |A_t f|^2 \ud t\Big)^{1/2} \Big \|_{p }  \lc \la^{-1/2}  (B_p(2\la)+C_p+1) \|f\|_p.
\end{equation}  

We wish to replace $A_t$ by the convolution operator with $(\chi \sigma)_t$ where $\chi$ has small compact support. To this end, for $g \in \Eann(\lambda/4, \la)$ we write
\[
(\chi\sigma)_t* g(x)=(2\pi)^{-2}
\int t^2\widehat \chi(t\eta) e^{i\inn {x}{\eta}}
\sigma_t* (g e^{-i\inn{\cdot}{\eta}}) (x) \ud \eta.
\]
Observe 
that $t^2\widehat \chi(t\eta)\lc_N (1+|\eta|)^{-N}$ {for $t \in [5/4,7/4]$}. Hence, {for $p > 2 $,}
\begin{align*} 
&\Big \|  \Big(\int_{5/4} ^{7/4} |(\chi\sigma)_t* g|^2 \ud t\Big)^{1/2} \Big \|_{p } \\
  &\qquad \lc
 \int_{|\eta|\le \la/8} (1+|\eta|)^{-N} \Big\|
 \Big(\int_{5/4} ^{7/4} |A_t [g e^{-i\inn{\cdot}{\eta}}]|^2 \ud t\Big)^{1/2} \Big \|_{p } \ud \eta  \\
\\& \qquad \qquad + 
\int_{|\eta|\ge \la/8} (1+|\eta|)^{-N} 
 \Big(\int_{5/4} ^{7/4} \|A_t [g e^{-i\inn{\cdot}{\eta}}]\|_p^2 \ud t\Big)^{1/2} \ud \eta.
 \end{align*}
For the first integral, we observe that
for $|\eta|\le \la/8 $ and $g\in \Eann(\la/4, \la)$ the modulated function $g e^{-i\inn{\cdot}{\eta}}$ belongs to $\Eann(\la/8, 9\la/8)$, and thus  one can apply \eqref{eq:2lambda} with $f=g e^{-i\inn{\cdot}{\eta}}$. For the second integral we get a decay factor of $O(\la^{1-N})$ from the $\eta$-integration.
This leads to the claimed inequality \eqref{eqn:localizedLpL2}.
\end{proof}

In what follows we use the differential notation for convolution operators  that are given by a    Fourier  multiplier $m$, i.e. 
$m(D) f= \cF^{-1}[m\widehat f]$.

\begin{lemma} \label{lem:aux} 
Let $u_1$, $u_2$ be $C^\infty_c(\bbR)$ functions  supported in $[-1,1]$ and let $u=u_1*u_2$.
Then, for any $\lambda \geq 1$,  $p>2$ and any integer $N \geq 0$,
\[\|u(|D|-\la)\|_{L^p(\bbR^2)\to L^p(\bbR^2) }  \le C(u_1,u_2,p,N) (1+\sup_{\rho\ge 1} \rho^{-N} B_p(\rho \la)).
\]
\end{lemma}

\begin{proof}
Let $\chi_1$, $\chi$ be $C^\infty_c$ functions, supported in a narrow sector and a neighborhood of a unit vector, so that $\chi(x)=1$ on the support of $\chi_1$. Let $Pg = \chi_1(\frac{D}{|D|})g$, which is a Fourier localization of $g$ to a sector.
By Lemma \ref{lem:LpL2chi}, 
\Be\notag  \Big\|  \Big(\int_{5/4}^{7/4} |(\chi \sigma)_t* Pg |^2 \ud t \Big)^{1/2} \Big\|_p  \lc \lambda^{-1/2} (B_p(2\la)+1) \|Pg\|_p \text{ for $g\in \Eann(\tfrac \la 4,\la)$.} \Ee

We use the method of stationary phase 
to get the usual asymptotics of the Fourier transform of $\chi\sigma$ in the conic support of $\chi_1(\xi/|\xi|)$. 
We obtain 
\[(\chi\sigma)_t *P g(x)=  ct^{-1/2} \cF^{-1} [\chi(\tfrac{\xi}{|\xi|}) |\xi|^{-1/2} e^{-it|\xi|} \widehat {Pg}](x) 
+R_{1} g(x,t)\]
where  the remainder term $R_{1}$ is a smoothing operator 
of order $-3/2$ satisfying  for $g\in \Eann(\tfrac \la 4,\la)$ the (negligible) bound
\[\|R_{1} g (\cdot, t) \|_{L^p} \lc \la^{-1}  \| g\|_p, \quad 1\le p\le \infty.\]
 Since $\chi(\tfrac{D} {|D|}) Pg= Pg$ (by the support properties of $\chi_1$ and $\chi$)  we get 
\Be \label{eqn:wavemicrolocal}   \Big\|  \Big(\int_{5/4}^{7/4}|e^{-it|D|}  Pg |^2 \ud t \Big)^{1/2} \Big\|_p  \lc  (B_p(2\la)+1) \|Pg\|_p, \text{ for $g\in \Eann(\la/4, \la)$. }\Ee

It will be convenient to switch to an inequality which involves an integral over $\bbR$ instead of $[5/4,7/4]$. We first look at constributions for $|t|\leq 2^{10}\lambda^{-1}$.
If $\zeta\in C^\infty$ is supported in $(2^{-5}, 2^5) $ then it is not hard to see that the multiplier $e^{it|\xi|}\zeta(\la^{-1}   |\xi|)$ is the Fourier transform of an $L^1$ function with $L^1$ norm uniformly bounded in $|t|\le 2^{10}\la^{-1}$.  Hence, for $p\ge 2$
\begin{multline*}  \Big\|  \Big(\int_{-2^{10}\la^{-1}}^{2^{10}\la^{-1} }
|e^{-it|D|}  Pg |^2 \ud t \Big)^{1/2} \Big\|_p  
\\
\lc
  \Big(\int_{-2^{10}\la^{-1}}^{2^{10}\la^{-1} }
   \|e^{-it|D|}  Pg \|^2_p \ud t \Big)^{1/2} \lc  \la^{-1/2} 
\|g\|_p\end{multline*} 
provided that $g\in \Eann(\la/4,\la)$.

{Next, we look at contributions for $|t| \geq \lambda^{-1}$. }
Let $a\in C^\infty$ supported in $\{\xi:1/4<|\xi|<4\} $ and let $\widetilde a(\xi)=a(\xi)\chi_1(\xi/|\xi|)$. Then for 
 $R\geq \la^{-1}$ we have 
\begin{multline*}
 \Big(\int_{5R/4}^{7R/4} |e^{-it|D|}  a(\la^{-1}D)Pg(x) |^2 \frac{\ud t}{t}  \Big)^{1/2}
 \\=\Big(\int_{5/4}^{7/4} |e^{-is|D|}  \widetilde a(R^{-1}\la^{-1} D ) [g(R\cdot)] (R^{-1}x)  |^2 \frac{\ud s}{s} \Big)^{1/2}.
\end{multline*} Thus, by scaling, we get from \eqref{eqn:wavemicrolocal} that 
\[  \Big\|  \Big(\int_{5R/4}^{7R/4} |e^{-it|D|}  a(\la^{-1}D)Pg |^2 \ud t \Big)^{1/2} \Big\|_p  \lc R^{1/2} (B_p(2R\la)+1) 
\|g\|_p
\]
for all $R\ge \la^{-1}$. 
Changing $t$ to $-t$ yields a similar inequality for  the interval $[-7R/4,-5R/4]$. Combining these estimates, for any Schwartz function  
$\vth$ 
and $N_1>0$ we obtain
\begin{multline*} \label{eqn:globalintegral} \Big\|  \Big(\int  |\vth(t) e^{it|D|}  Pg |^2 \ud t \Big)^{1/2} \Big\|_p \\ \lc C(\vth) 
\Big(\la^{-1/2}  +  \sum_{\la^{-1}\le 2^k\le 1} 2^{k/2} B_p(\la 2^{k+1}) +\sum_{k\ge 0} 2^{k(\frac 12-N_1)} B_p(\la 2^{k+1}) \Big)
\|g\|_p\end{multline*}
provided that $g\in \Eann(\la/4, \la)$. {For the second sum we have used the rapid decay of $\vth$ for $|t| \geq 1$.}
This implies,  for $g\in \Eann(\la/4, \la)$ and $N \geq 0$ that
\Be \label{eqn:globalintegral} \Big\|  \Big(\int  |\vth(t) e^{it|D|}  Pg |^2 \ud t \Big)^{1/2} \Big\|_p \lc_N( 1+ \sup_{\rho\ge 1}\rho^{-N} B_p(\rho\la)),
\Ee
{using that $B_p(\lambda 2^{k+1}) \leq B_p(2\lambda)$ for $\lambda^{-1} \leq 2^k \leq 1$.}
We now come to the inequality asserted in the lemma. 
Note that 
\[ u(|D|-\la) P g(x)= \frac{1}{2\pi} \int\widehat u_2(\tau) \widehat u_1(\tau) e^{-i\la \tau} e^{i\tau|D|} P g(x) \ud\tau,\]
and by the Cauchy--Schwarz inequality
\[ | u(|D|-\la) P g(x)| \lc
\Big( \int| \widehat u_1(\tau)    e^{i\tau|D|} P g(x) |^2 \ud\tau\Big)^{1/2}.\]
If we apply \eqref{eqn:globalintegral} with $\vth(\tau) = \widehat 
u_1(\tau)$ we get the inequality asserted in the lemma,
first for  functions $g \in \Eann (\la/4, \la)$ but by the support property of $u$  it is implied for   general $g\in L^p$.
\end{proof}

A  last lemma deals with lower bounds for such multipliers.

\begin{lemma}\label{lem:aux2} There exist $\eps>0$, $\la_1=\la_1(p)>1$ and $c>0$  such that for all $\la>\la_1$ the following holds for $p>2$: 

For all nonnegative $L^\infty$ functions $u$ supported in $[-2\eps^2,2\eps^2]$ that are bounded below by $1$ in $[-\eps^2,\eps^2]$,  and all integers $n\geq 10$,
\Be\label{eq:lowerbd-Cord}\| u(|D|-2^{2n})\|_{L^p(\bbR^2)\to L^p(\bbR^2)}  \ge c  n^{\frac 12-\frac 1p}\,.
\Ee
\end{lemma}
This is proved by a variant of Fefferman's proof for the ball multiplier \cite{FeDisc}, using the Besicovitch construction, together with a standard randomization argument.
It is known but not well-documented that Fefferman's proof also gives lower bounds for multipliers such as in  Lemma \ref{lem:aux2}; in fact the second-named author had presented a version of the lemma in a graduate course at the University of Chicago in 1985.  
Because of the lack of an appropriate reference in the literature in the precise form needed here, we give the proof for the convenience of the reader.
Other applications of Fefferman's argument have been used, for example,  in Fourier restriction theory \cite{BCSS}, \cite{Craig25}, for resolvent bounds for certain partial differential equations \cite{KenigTomas-TAMS1980}, \cite{Ruiz-PAMS1983-1} and recently in a local theory for cone multipliers with applications to  Cauchy--Szeg\H o projections \cite{BallestaGarrigos}.
 We remark  that the 
 lower bound \eqref{eq:lowerbd-Cord} matches for $2\le p\le 4$ the upper bound 
\[\| u(|D|-2^{2n})\|_{L^p(\bbR^2)\to L^p(\bbR^2)}  \le C n^{\frac 12-\frac 1p}, \quad 
2\le p\le 4,
\]
for which one requires an  additional regularity assumption, say    $u\in C^2$.  It follows 
from C\'ordoba's work \cite{Cordoba1979}.

\begin{proof}[Proof of Lemma \ref{lem:aux2}]
We use the construction by Keich \cite{Keich} which gives a slightly better upper bound in the construction of Besicovitch sets than the one used in, say, \cite{FeDisc}. Consider line segments $\ell =\{ (s, as+b), s\in [0,1]\}$ where $a=a(\ell)\in [0,1]$ and $b=b(\ell)\in [-1,0]$.  We write $\ell(s):=a(\ell)s+b(\ell)$. For given large $n\in \bbN$  let $T(\ell)\equiv T^n(\ell)$ be the triangle with vertices $(0,\ell(0))$, $(0,\ell(0)-2^{-n})$, $(1, \ell(1))$.
Let $\vec T(\ell)$ be the {\it reach of $T(\ell)$}, defined to be  the triangle obtained by translating $T(\ell)$ by $2\sqrt 2$ along the direction of $\ell$.

Fix $n\in \bbN$, $n\ge 10$.  It is shown in  \cite{Keich} that there exists a collection of line segments $\{\ell_\nu\}_{\nu=0}^{2^{n}-1}$ with $a_\nu \equiv a(\ell_\nu)=\nu 2^{-n}$ such that the triangles $T(\ell_\nu)$ satisfy 
\Be\label{eqn:measBesic}\meas \big (\bigcup_{\nu=0}^{2^n-1} T(\ell_\nu)\big)
< n^{-1} \Ee
and the corresponding reaches $\vec T(\ell_\nu)$ are pairwise disjoint.


For each $\nu=0,\dots, 2^{n}-1$, let
\[ e_\nu=\frac {(1,a_\nu)}{\sqrt{1+a_\nu^2}}, \qquad
e_\nu^\perp =\frac {(-a_\nu,1)}{\sqrt{1+a_\nu^2}} 
\]
and consider the function
$$f_\nu(y) = \bbone_{T(\ell_\nu)}(y) e^{i \inn{y}{2^{2n}e_\nu}}.
$$
Let $\eps>0$ be sufficiently small, chosen to satisfy the requirements of the forthcoming argument. Let $\psi \in C_c^\infty$ be non-negative, supported in $(-1/2,1/2)$ and bounded below by $1$ in $(-1/4,1/4)$.
Define  \[ h_{\nu}(\xi_2) = \psi \big(2^{-n}\eps^{-1} (\xi_2- 2^{2n}\tfrac{a_\nu}{\sqrt{1+a_\nu^2}} )\big), \quad \ka_\nu(x_2) = \frac{1}{2\pi} \int  h_{\nu}(\xi_2) e^{ix_2\xi_2} \ud\xi_2\] 
and let $L_\nu$ denote the operator given by $L_\nu g:= \kappa_\nu \ast_2 g$, where $\ast_2$ denotes the convolution in the second variable. We will first show the lower bound
\Be \label{eq:Knufnu pre}
|u(|D|-2^{2n}) L_\nu f_\nu (x)|
 \ge c, \quad x\in \vec T(\ell_\nu),
\Ee
for some $c>0$. To this end, let $K_\nu$ denote the convolution kernel of the operator $u(|D|-2^{2n}) L_\nu$, given by
\Be \label{eqn:Knudef}
K_{\nu}(x) =(2\pi)^{-2} \int u( |\xi|- 2^{2n} )h_{\nu}(\xi_2) e^{i\inn {x}{\xi}} \ud\xi.
\Ee
Then 
\begin{multline*}K_\nu(x) e^{-i \inn{x}{2^{2n} e_\nu}} =\\
(2\pi)^{-2} \int u (|\xi|- 2^{2n}) h_{\nu}(\xi_2) e^{i\inn {x}{e_\nu} \inn{e_\nu}{\xi-2^{2n} e_\nu}} 
e^{i\inn {x}{e_\nu^\perp} \inn{e_\nu^\perp} {\xi}} \ud \xi
\end{multline*}
and a computation shows that on the support of integration we have
\begin{equation}\label{eq:claim}
    |\inn{e_\nu}{\xi-2^{2n} e_\nu}|\le 2^{6} \eps^2
\quad \text{ and } \quad   |\inn {\xi}{e_\nu^\perp}| \le 2^{n+2}\eps.
\end{equation}
To see this, let $\xi=(\xi_1,\xi_2)$ satisfying $\big||\xi|-2^{2n} \big|\le 2\eps^2$  and 
$|\xi_2-2^{ 2n} \frac{a_\nu}{(1+a_\nu^2)^{1/2}}|\le 2^{n-1}\eps$. Set $\eta=2^{-2n} \xi$, so that $|\eta|=1+\varrho$ with $|\varrho|\le \eps^2 2^{-2n+1} $ and 
$\eta_2= \frac{a_\nu}{\sqrt{1+a_\nu^2}} +v$ with $|v|\le \eps 2^{-n -1}$. 
We show that $|\eta_1- \frac{1}{\sqrt{1+a_\nu^2}}|\le \eps 2^{-n+1}. $
Write \begin{align*} 
\eta_1&= \sqrt{|\eta|^2-\eta^2_2}= \big((1+\varrho)^2-(\tfrac{a_\nu}{\sqrt{1+a_\nu^2}} +v)^2\big)^{1/2}
\\
&= \big( 1- \tfrac{a_\nu^2}{1+a_\nu^2}  +2\varrho+\varrho^2-2\tfrac{a_\nu}{\sqrt{1+a_\nu^2}} v-v^2 \big)^{1/2}
= \big( \tfrac{1}{1+a_\nu^2} +\Delta \big)^{1/2} 
\end{align*} where $\Delta= 2\varrho+\varrho^2-2\tfrac{a_\nu}{\sqrt{1+a_\nu^2}} v-v^2$ and thus $|\Delta|\le  \eps^2 2^{2n+2}+\eps^4 2^{-4n+2} + \eps 2^{-n} +\eps^22^{-2n-2}\le \eps 2^{-n+1}.$
Hence
\[ |\eta_1- \tfrac{1}{\sqrt{1+a_\nu^2}}  |=\big|
\big( \tfrac{1}{{1+a_\nu^2}} +\Delta \big)^{1/2} -\big( \tfrac{1}{1+a_\nu^2} \big)^{1/2} \big|\le |\Delta|\le \eps 2^{-n+1} \]
and then also 
\[ |\eta-e_\nu| \le  |\eta_1- \tfrac{1}{\sqrt{1+a_\nu^2}}|+ |\eta_2 -\tfrac{a_\nu}{\sqrt{1+a_\nu^2}} | \le \eps 2^{-n+2}. \]
Next write $\eta = e_\nu+\om_1e_\nu+\om_2 e_\nu^\perp$ and observe that $\sqrt{\om_1^2+\om_2^2} =|\eta-e_\nu|$, so $|\om|^2\le (\eps 2^{-n+2})^2$. We have
$\inn{e_\nu}{\eta-e_\nu} = \inn{e_\nu}{\om_1e_\nu +\om_2 e_\nu^\perp }=\om_1$.
Moreover
\begin{align*} |\eta|-1 &= |e_\nu+\om_1e_\nu+\om_2e_\nu^\perp|-1 = \big((1+\om_1)^2+\om_2^2\big)^{1/2} -1
\\
&=
 (1+2\om_1+|\om|^2)^{1/2}-1  =\om_1+ \tfrac{|\om|^2} 2+ E(\om)
\end{align*} and using $|(1+s)^{1/2}-1- \tfrac s2|\le \frac 14(1-|s|)^{-3/2}\frac {s^2}2  $, which follows from Taylor's expansion on both sides of the inequality, we estimate the error by 
 $|E(\om)| \le \tfrac{1}8\tfrac{(2|\om_1|+|\om|^2)^2}
{(1-2|\om_1|-|\om|^2)^{3/2} }\le |\om|^2$   (recall  $|\om|^2\le (\eps 2^{-n+2})^2$ and $n\ge 10$).  
Hence, since $|\eta|=1+\varrho$ with $|\varrho|\le \eps^22^{-2n+1} $, 
\[|\om_1| \le \big||\eta|-1\big| +  2|\om|^2 = \big||\eta|-1\big|+ 2|\eta-e_\nu|^2  \le 2^{-2n+1}\eps^2+ 2^{-2n+5} \eps^2\] and thus $|\om_1|\le 2^{-2n+6}\eps^2$. 
We get \begin{align*} &|\inn{e_\nu}{\eta-e_\nu}|=|\om_1|\le 2^{-2n+6}\eps^2,
\\
&|\inn {\eta}{e_\nu^\perp}|= |\inn {\eta-e_\nu}{e_\nu^\perp}| \le |\eta-e_\nu| \le 2^{-n+2}\eps,
\end{align*} 
and from this 
 $|\inn{e_\nu}{\xi-2^{2n} e_\nu}|\le 2^{6}\eps^2$
and $|\inn {\xi}{e_\nu^\perp}| \le 2^{n+2}\eps$, which correspond to the claimed bounds \eqref{eq:claim}.

Hence, choosing $\varepsilon$ sufficiently small,
\Be\notag
\Re\big( e^{-i 2^{2n}\inn{x}{e_\nu}} K_\nu(x) \big) \ge c\eps^3 2^{n} \text{ if } |\inn{x}{e_\nu}|\le 2^4 \text{ and }  |\inn{x}{e_\nu^\perp}|\le  2^{-n+4}.
\Ee
As a consequence  we get the lower bound \eqref{eq:Knufnu pre} on the reach of $T(\ell_\nu)$, namely 
\Be \label{eqn:Knufnu} 
|K_\nu* f_\nu (x)|
=\Big| \int K_\nu(x-y) e^{-i\inn{x-y}{2^{2n} e_\nu} } \bbone_{T(\ell_\nu)}(y) \ud y \Big|
 \ge c, \quad x\in \vec T(\ell_\nu).
\Ee

Now, define for $\om\in [0,1]$ 
\Be \notag f^\om(y) =\sum_{\nu=0}^{2^n-1} r_\nu(\om) \kappa_\nu \ast_2 f_\nu
\Ee
where $(r_\nu)_{\nu\in \bbN}$ is the sequence of Rademacher functions. If \[\cC_{p,n}=\big\|u(|D|-2^{2n})\big\|_{L^p\to L^p}\] we have by duality 
\[
 \|u(|D|-2^{2n}) f^\om\|_{p'} \le \cC_{p,n} \|f^\om\|_{p'}, \quad \om\in [0,1].
 \] Integrating in $\omega$ and using the above definitions we get
 \Be \notag
\Big(\int_0^1 \Big\|\sum_{\nu=0}^{2^n-1} r_\nu(\om) K_\nu*f_\nu \Big\|_{p'}^{p'} \ud \om \Big)^{1/p'} 
\le \cC_{p,n} \Big(\int_0^1 \Big\|\sum_{\nu=0}^{2^n-1} r_\nu(\om) \ka_\nu*\ci{2} f_\nu\Big\|_{p'}^{p'} \ud \om\Big)^{1/p'}.
\Ee
We interchange the $x$ and $\omega$ integration on both sides and using both lower and upper bounds in Khinchine's inequality (see e.g. \cite[Appendix D]{Stein70SI}) we get
\Be\label{eqn:Kh}
\Big\|\Big(\sum_{\nu=0}^{2^n-1} |
K_\nu*f_\nu |^2\Big)^{1/2}\Big\|_{p'} \le C(p) \cC_{p,n} 
\Big\|\Big(\sum_{\nu=0}^{2^n-1} |
\ka_\nu*\ci{2}f_\nu |^2\Big)^{1/2}\Big\|_{p'} .
\Ee
We first give a lower bound for the left hand side of \eqref{eqn:Kh}. By the disjointness of the $\vec T(\ell_\nu)$, and \eqref{eqn:Knufnu} 
\begin{align} \notag
&\Big\|\Big(\sum_{\nu=0}^{2^n-1} |
K_\nu*f_\nu |^2\Big)^{1/2}\Big\|_{p'} 
\ge \Big(\sum_{\nu'=0}^{2^n-1} 
\int_{\vec T(\ell_{\nu'})} \Big(\sum_{\nu=0}^{2^n-1} |
K_\nu*f_\nu |^2\Big)^{p'/2}
\ud x\Big)^{1/p'}
\\ \label{eqn:lowerbdkak}
&\ge \Big(
\sum_{\nu'=0}^{2^n -1} 
\int_{\vec T(\ell_{\nu'})} |K_{\nu'}*f_{\nu'} |^{p'} \ud x\Big)^{1/p'}
\ge c \, \Big(\sum_{\nu=0}^{2^{n}-1} |\vec T(\ell_\nu)|\Big)^{1/p'} \ge 2^{-1/p'} c.
\end{align}

We give an upper bound for the right hand side of 
\eqref{eqn:Kh}.
Use the uniform pointwise bound
\[|\ka_\nu(x_2) |\le C2^n (1+2^n|x_2|)^{-2}\] and the fact that all $f_\nu$ are supported in $\cup_{\nu=0}^{2^n-1}T(\ell_\nu)  $ which by \eqref{eqn:measBesic} is a set of measure $<1/n$. It follows  
\begin{align} \notag & \Big\|\Big(\sum_{\nu=0}^{2^n-1} |
\ka_\nu*_2f_\nu |^2\Big)^{1/2}\Big\|_{p'}
\le C
\Big\|\Big(\sum_{\nu=0}^{2^n-1} |
f_\nu |^2\Big)^{1/2}\Big\|_{p'}
\\ \notag
&\le C \,  \meas\Big(\bigcup_{\nu=0}^{2^n-1} T(\ell_\nu) \Big)^{1/p'-1/2}\Big\|\Big(\sum_{\nu=0}^{2^n-1} |
f_\nu |^2\Big)^{1/2}\Big\|_{2}
\\   \label{eqn:upperbdkak} &
\le Cn^{-1/p'+1/2}\Big(\sum_{\nu=0}^{2^n-1} |T(\ell_\nu)|\Big)^{1/2}\le C 2^{-1/2} n^{1/p-1/2}.
\end{align}
Combining \eqref{eqn:Kh}, \eqref{eqn:lowerbdkak} and \eqref{eqn:upperbdkak}  we get 
$$ c2^{-1/p'}  \le C2^{-1/2}  C(p) \cC_{p,n}  n^{1/p-1/2} $$ and thus the assertion of the lemma.
\end{proof}

\begin{proof}[Proof of Theorem \ref{thm:kakeya}, conclusion] 
By a scaling argument we can replace $2^{2n}$ in Lemma \ref{lem:aux2} with $\la\in [2^{2n}, 2^{2n+2}]$. From    Lemma \ref{lem:aux} and Lemma \ref{lem:aux2}  it follows that there is a $\mu_0>2$, $c_0>0$ such that for all $\mu>\mu_0$
\[\sup_{\rho\ge 1} B_p(\rho\mu)\rho^{-N} \ge c_0 (\log \mu)^{1/2-1/p}.\]
The trivial bound \eqref{eq:finiteness Bp} implies $B_p(\la)\le C\la$ and hence 
\[\sup_{\rho\ge \mu } B_p(\rho\mu)\rho^{-N} \le C_N\mu  \sup_{\rho\ge \mu}  \rho^{1-N} \le C_N \mu^{2-N}.\]
Then  
\begin{align*}
    B_p(\mu^2) & \ge \sup_{1\le \rho\le \mu } B_p(\rho\mu)\rho^{-N} 
\ge \sup_{\rho\ge 1 } B_p(\rho\mu)\rho^{-N} 
-C_N \mu^{2-N}\\
& \ge c_0 (\log\mu )^{1/2-1/p} -C_N \mu^{2-N}
\end{align*}
and thus we get for $\mu>\mu_1=\max\{\mu_0, \exp( (2C_2/c_0)^{\frac{2p}{p-2}})\}$
\[ B_p(\mu^2)\ge \tfrac{c_0}{2} (\tfrac 12\log \mu^2)^{\frac 12-\frac 1p}\]
which implies the theorem.
\end{proof} 

\begin{remark} Let $\fA f(x,t)= \chi(t) A_tf(x)$ where $\chi$ is a nontrivial bump function  compactly supported in $(1,2)$. An examination of our proof (in particular the proof of Lemma 
\ref{lem:LpL2chi}) 
also shows that for $p>2$
\begin{equation} \label{eq:Besinfty}  \sup_{f\in \cS(\bbR^2) } \sup_{\|f\|_p\le 1} \| \fA  f\|_{L^p(\dot B^{1/2}_{2,\infty} )} =\infty. 
\end{equation} 
In view of the embedding $V_2\hookrightarrow \dot B^{1/2}_{2,\infty} $ this gives a stronger lower bound than stated in Theorem \ref{thm:main}. Note that $V_2$ embeds into $L^\infty$ while $B^{1/2}_{2,\infty}$ does not.  Moreover, we may strengthen this formulation of our result by replacing in \eqref{eq:Besinfty} the $B^{1/2}_{2,\infty}$ norm or $\dot B^{1/2}_{2,\infty} $ semi-norm of $a(t)= \fA(x,t)$  with 
\[\sup_{n>0}  \|\La_n a\|_2 2^{n/2} \om_n, \] 
where  $\omega_n$  may be  small for large $n$, such that $\om_n=n^{\frac 1p-\frac 12+\eps}$;  in fact we can choose any sequence 
 satisfying  
  $
\limsup_{n\to \infty} n^{1/2-1/p} \om_n =\infty.$
\end{remark}

\subsection*{\it Acknowledgements}  D.B. is supported by the grants RYC2020-029151-I and PID2022-140977NA-I00,  funded by MICIU/AEI/10.13039/501100011033, by ``ESF Investing in your future" and by FEDER, UE. L. R. is supported by the grants CEX2021-001142-S and PID2023-146646NB-I00 funded by MICIU/AEI/10.13039/501100011033 and
by ESF+, by BERC 2022-2025 of the Basque Government and IKERBASQUE. A.S. is supported in part by NSF grant 2348797.  
\bibliographystyle{plain} 
\bibliography{Reference}

@article {Bourgain1986,
    AUTHOR = {Bourgain, Jean},
     TITLE = {Averages in the plane over convex curves and maximal operators},
   JOURNAL = {J. Analyse Math.},
  FJOURNAL = {Journal d'Analyse Math\'ematique},
    VOLUME = {47},
      YEAR = {1986},
     PAGES = {69--85},
      ISSN = {0021-7670},
   MRCLASS = {42B25 (52A10)},
  MRNUMBER = {874045},
MRREVIEWER = {K. J. Falconer},
       URL = {https://doi.org/10.1007/BF02792533},
}

@article {Stein1976,
    AUTHOR = {Stein, Elias M.},
     TITLE = {Maximal functions. {I}. {S}pherical means},
   JOURNAL = {Proc. Nat. Acad. Sci. U.S.A.},
  FJOURNAL = {Proceedings of the National Academy of Sciences of the United
              States of America},
    VOLUME = {73},
      YEAR = {1976},
    NUMBER = {7},
     PAGES = {2174--2175},
      ISSN = {0027-8424},
   MRCLASS = {42A40 (43A85)},
  MRNUMBER = {0420116},
MRREVIEWER = {Alberto Torchinsky},
}

@article {JSW,
    AUTHOR = {Jones, Roger L. and Seeger, Andreas and Wright, James},
     TITLE = {Strong variational and jump inequalities in harmonic analysis},
   JOURNAL = {Trans. Amer. Math. Soc.},
  FJOURNAL = {Transactions of the American Mathematical Society},
    VOLUME = {360},
      YEAR = {2008},
    NUMBER = {12},
     PAGES = {6711--6742},
      ISSN = {0002-9947},
   MRCLASS = {42B20 (42B15 42B25)},
  MRNUMBER = {2434308},
MRREVIEWER = {Javier Duoandikoetxea},
       DOI = {10.1090/S0002-9947-08-04538-8},
       URL = {https://doi.org/10.1090/S0002-9947-08-04538-8},
}

@article {ChristPAMS,
    AUTHOR = {Christ, Michael},
     TITLE = {On almost everywhere convergence of {B}ochner--{R}iesz means in
              higher dimensions},
   JOURNAL = {Proc. Amer. Math. Soc.},
  FJOURNAL = {Proceedings of the American Mathematical Society},
    VOLUME = {95},
      YEAR = {1985},
    NUMBER = {1},
     PAGES = {16--20},
      ISSN = {0002-9939},
     CODEN = {PAMYAR},
   MRCLASS = {42B25 (47G05)},
  MRNUMBER = {796439 (87c:42020)},
MRREVIEWER = {Shan Zhen Lu},
       DOI = {10.2307/2045566},
       URL = {http://dx.doi.org/10.2307/2045566},
}

@book {Stein70SI,
    AUTHOR = {Stein, Elias M.},
     TITLE = {Singular integrals and differentiability properties of
              functions},
    SERIES = {Princeton Mathematical Series, No. 30},
 PUBLISHER = {Princeton University Press, Princeton, N.J.},
      YEAR = {1970},
     PAGES = {xiv+290},
   MRCLASS = {46.38 (26.00)},
  MRNUMBER = {0290095 (44 \#7280)},
MRREVIEWER = {R. E. Edwards},
}

@Article{FeDisc,
  author     = {Fefferman, Charles},
  title      = {The multiplier problem for the ball},
  journal    = {Ann. of Math. (2)},
  year       = {1971},
  volume     = {94},
  pages      = {330--336},
  fjournal   = {Annals of Mathematics. Second Series},
  issn       = {0003-486X},
  mrclass    = {42A18 (42A92 47B99)},
  mrnumber   = {0296602 (45 \#5661)},
  mrreviewer = {R. Larsen},
}

@Article{BCSS,
  author     = {Beckner, William and Carbery, Anthony and Semmes, Stephen and Soria, Fernando},
  title      = {A note on restriction of the {F}ourier transform to spheres},
  journal    = {Bull. London Math. Soc.},
  year       = {1989},
  volume     = {21},
  number     = {4},
  pages      = {394--398},
  coden      = {LMSBBT},
  doi        = {10.1112/blms/21.4.394},
  fjournal   = {The Bulletin of the London Mathematical Society},
  issn       = {0024-6093},
  mrclass    = {42B10},
  mrnumber   = {998638 (90i:42023)},
  mrreviewer = {Satoru Igari},
  url        = {http://dx.doi.org/10.1112/blms/21.4.394},
}

@article {Lepingle1976,
    AUTHOR = {L\'{e}pingle, Dominique},
     TITLE = {La variation d'ordre {$p$} des semi-martingales},
   JOURNAL = {Z. Wahrscheinlichkeitstheorie und Verw. Gebiete},
  FJOURNAL = {Zeitschrift f\"{u}r Wahrscheinlichkeitstheorie und Verwandte
              Gebiete},
    VOLUME = {36},
      YEAR = {1976},
    NUMBER = {4},
     PAGES = {295--316},
   MRCLASS = {60G45},
  MRNUMBER = {420837},
MRREVIEWER = {Norihiko Kazamaki},
       DOI = {10.1007/BF00532696},
       URL = {https://doi-org.ezproxy.library.wisc.edu/10.1007/BF00532696},
}

@article {BerghPeetre,
    AUTHOR = {Bergh, J\"{o}ran and Peetre, Jaak},
     TITLE = {On the spaces {$V_{p}$} {$(0<p\leq \infty )$}},
   JOURNAL = {Boll. Un. Mat. Ital. (4)},
    VOLUME = {10},
      YEAR = {1974},
     PAGES = {632--648},
   MRCLASS = {46E35 (41A15)},
  MRNUMBER = {0380389},
MRREVIEWER = {A. Kufner},
}

@incollection {LeeRogersSeeger2014,
    AUTHOR = {Lee, Sanghyuk and Rogers, Keith M. and Seeger, Andreas},
     TITLE = {Square functions and maximal operators associated with radial
              {F}ourier multipliers},
 BOOKTITLE = {Advances in analysis: the legacy of {E}lias {M}. {S}tein},
    SERIES = {Princeton Math. Ser.},
    VOLUME = {50},
     PAGES = {273--302},
 PUBLISHER = {Princeton Univ. Press, Princeton, NJ},
      YEAR = {2014},
   MRCLASS = {42A20 (42B15 42B25)},
  MRNUMBER = {3329855},
MRREVIEWER = {Jan-Olav R\"{o}nning},
}

@article {KanekoSunouchi,
    AUTHOR = {Kaneko, Makoto and Sunouchi, Gen-ichir\^{o}},
     TITLE = {On the {L}ittlewood--{P}aley and {M}arcinkiewicz functions in
              higher dimensions},
   JOURNAL = {Tohoku Math. J. (2)},
  FJOURNAL = {The Tohoku Mathematical Journal. Second Series},
    VOLUME = {37},
      YEAR = {1985},
    NUMBER = {3},
     PAGES = {343--365},
      ISSN = {0040-8735},
   MRCLASS = {42B25},
  MRNUMBER = {799527},
MRREVIEWER = {G. V. Welland},
       DOI = {10.2748/tmj/1178228647},
       URL = {https://doi-org.ezproxy.library.wisc.edu/10.2748/tmj/1178228647},
}

@article {Carbery1983,
    AUTHOR = {Carbery, Anthony},
     TITLE = {The boundedness of the maximal {B}ochner--{R}iesz operator on
              {$L^{4}({\mathbb R}^{2})$}},
   JOURNAL = {Duke Math. J.},
  FJOURNAL = {Duke Mathematical Journal},
    VOLUME = {50},
      YEAR = {1983},
    NUMBER = {2},
     PAGES = {409--416},
      ISSN = {0012-7094},
   MRCLASS = {42B15 (42B25)},
  MRNUMBER = {705033},
MRREVIEWER = {Jos\'{e} L. Rubio de Francia},
       URL = {http://projecteuclid.org.ezproxy.library.wisc.edu/euclid.dmj/1077303202},
}

@article {Seeger-crelle1986,
    AUTHOR = {Seeger, Andreas},
     TITLE = {On quasiradial {F}ourier multipliers and their maximal
              functions},
   JOURNAL = {J. Reine Angew. Math.},
  FJOURNAL = {Journal f\"{u}r die Reine und Angewandte Mathematik. [Crelle's
              Journal]},
    VOLUME = {370},
      YEAR = {1986},
     PAGES = {61--73},
      ISSN = {0075-4102},
   MRCLASS = {42B25 (42B15)},
  MRNUMBER = {852510},
MRREVIEWER = {Satoru Igari},
       DOI = {10.1515/crll.1986.370.61},
       URL = {https://doi-org.ezproxy.library.wisc.edu/10.1515/crll.1986.370.61},
}

@article {Cordoba1979,
    AUTHOR = {C\'{o}rdoba, Antonio},
     TITLE = {A note on {B}ochner--{R}iesz operators},
   JOURNAL = {Duke Math. J.},
  FJOURNAL = {Duke Mathematical Journal},
    VOLUME = {46},
      YEAR = {1979},
    NUMBER = {3},
     PAGES = {505--511},
      ISSN = {0012-7094},
   MRCLASS = {42B15 (42B25)},
  MRNUMBER = {544242},
MRREVIEWER = {Richard Bagby},
       URL = {http://projecteuclid.org.ezproxy.library.wisc.edu/euclid.dmj/1077313571},
}

@article {Keich,
    AUTHOR = {Keich, Uri},
     TITLE = {On {$L^p$} bounds for {K}akeya maximal functions and the
              {M}inkowski dimension in {${\mathbb R}^2$}},
   JOURNAL = {Bull. London Math. Soc.},
  FJOURNAL = {The Bulletin of the London Mathematical Society},
    VOLUME = {31},
      YEAR = {1999},
    NUMBER = {2},
     PAGES = {213--221},
      ISSN = {0024-6093},
   MRCLASS = {42B25 (28A78)},
  MRNUMBER = {1664129},
MRREVIEWER = {Andreas Seeger},
       DOI = {10.1112/S0024609398005372},
       URL = {https://doi-org.ezproxy.library.wisc.edu/10.1112/S0024609398005372},
}

@book {Triebel1983,
    AUTHOR = {Triebel, Hans},
     TITLE = {Theory of function spaces},
    SERIES = {Monographs in Mathematics},
    VOLUME = {78},
 PUBLISHER = {Birkh\"{a}user Verlag, Basel},
      YEAR = {1983},
     PAGES = {284},
      ISBN = {3-7643-1381-1},
   MRCLASS = {46Exx},
  MRNUMBER = {781540},
       DOI = {10.1007/978-3-0346-0416-1},
       URL = {https://doi-org.ezproxy.library.wisc.edu/10.1007/978-3-0346-0416-1},
}

@article {BeltranOberlinRoncalSeegerStovall,
    AUTHOR = {Beltran, David and Oberlin, Richard and Roncal, Luz and
              Seeger, Andreas and Stovall, Betsy},
     TITLE = {Variation bounds for spherical averages},
   JOURNAL = {Math. Ann.},
  FJOURNAL = {Mathematische Annalen},
    VOLUME = {382},
      YEAR = {2022},
    NUMBER = {1-2},
     PAGES = {459--512},
      ISSN = {0025-5831,1432-1807},
   MRCLASS = {42B15 (42B25)},
  MRNUMBER = {4377310},
MRREVIEWER = {Sundaram\ Thangavelu},
       DOI = {10.1007/s00208-021-02218-2},
       URL = {https://doi-org.ezproxy.library.wisc.edu/10.1007/s00208-021-02218-2},
}

@article {BallestaGarrigos,
    AUTHOR = {Ballesta Yag\"ue, Fernando and Garrig\'os, Gustavo},
     TITLE = {Local cone multipliers and {C}auchy-{S}zeg\"o{} projections in
              bounded symmetric domains},
   JOURNAL = {J. Lond. Math. Soc. (2)},
  FJOURNAL = {Journal of the London Mathematical Society. Second Series},
    VOLUME = {110},
      YEAR = {2024},
    NUMBER = {4},
     PAGES = {Paper No. e12986, 20},
      ISSN = {0024-6107,1469-7750},
   MRCLASS = {32A25 (15B48 32A35 32M15 42B15 47B32)},
  MRNUMBER = {4801889},
MRREVIEWER = {Zhiming\ Feng},
       DOI = {10.1112/jlms.12986},
       URL = {https://doi-org.ezproxy.library.wisc.edu/10.1112/jlms.12986},
}

@article {Craig25,
    AUTHOR = {Craig, Sam},
     TITLE = {Failure of weak-type endpoint restriction estimates for
              quadratic manifolds.},
   JOURNAL = {J. Geom. Anal.},
  FJOURNAL = {Journal of Geometric Analysis},
    VOLUME = {35},
      YEAR = {2025},
    NUMBER = {12},
     PAGES = {Paper No. 382, 14},
      ISSN = {1050-6926, 1559-002X},
   MRCLASS = {42B25},
  MRNUMBER = {4970321},
       DOI = {10.1007/s12220-025-02209-8},
       URL = {https://doi-org.ezproxy.library.wisc.edu/10.1007/s12220-025-02209-8},
}

@article {SteinActa58,
    AUTHOR = {Stein, Elias M.},
     TITLE = {Localization and summability of multiple {F}ourier series},
   JOURNAL = {Acta Math.},
  FJOURNAL = {Acta Mathematica},
    VOLUME = {100},
      YEAR = {1958},
     PAGES = {93--147},
      ISSN = {0001-5962,1871-2509},
   MRCLASS = {42.00},
  MRNUMBER = {105592},
MRREVIEWER = {K.\ Chandrasekharan},
       DOI = {10.1007/BF02559603},
       URL = {https://doi-org.ezproxy.library.wisc.edu/10.1007/BF02559603},
}

@article {SteinWainger-Bull78,
    AUTHOR = {Stein, Elias M. and Wainger, Stephen},
     TITLE = {Problems in harmonic analysis related to curvature},
   JOURNAL = {Bull. Amer. Math. Soc.},
  FJOURNAL = {Bulletin of the American Mathematical Society},
    VOLUME = {84},
      YEAR = {1978},
    NUMBER = {6},
     PAGES = {1239--1295},
      ISSN = {0002-9904},
   MRCLASS = {42B20 (28A15)},
  MRNUMBER = {508453},
MRREVIEWER = {Alberto\ Torchinsky},
       DOI = {10.1090/S0002-9904-1978-14554-6},
       URL = {https://doi-org.ezproxy.library.wisc.edu/10.1090/S0002-9904-1978-14554-6},
}

@article {KenigTomas-TAMS1980,
    AUTHOR = {Kenig, Carlos E. and Tomas, Peter A.},
     TITLE = {{$L\sp{p}$}\ behavior of certain second order partial
              differential operators},
   JOURNAL = {Trans. Amer. Math. Soc.},
  FJOURNAL = {Transactions of the American Mathematical Society},
    VOLUME = {262},
      YEAR = {1980},
    NUMBER = {2},
     PAGES = {521--531},
      ISSN = {0002-9947,1088-6850},
   MRCLASS = {42B15 (35E20 42A45)},
  MRNUMBER = {586732},
MRREVIEWER = {Jos\'e\ Garc\'ia-Cuerva},
       DOI = {10.2307/1999843},
       URL = {https://doi-org.ezproxy.library.wisc.edu/10.2307/1999843},
}

@article {Ruiz-PAMS1983-1,
    AUTHOR = {Ruiz, Alberto},
     TITLE = {{$L\sp{p}$}-boundedness of a certain class of multipliers
              associated with curves on the plane. {I}},
   JOURNAL = {Proc. Amer. Math. Soc.},
  FJOURNAL = {Proceedings of the American Mathematical Society},
    VOLUME = {87},
      YEAR = {1983},
    NUMBER = {2},
     PAGES = {271--276},
      ISSN = {0002-9939,1088-6826},
   MRCLASS = {42B15},
  MRNUMBER = {681833},
MRREVIEWER = {Kenneth\ F.\ Andersen},
       DOI = {10.2307/2043701},
       URL = {https://doi-org.ezproxy.library.wisc.edu/10.2307/2043701},
}
\end{document}